\documentclass[11pt]{article}
\usepackage[british]{babel}
\usepackage{amssymb,amsmath,latexsym,amsthm,amsfonts}
\usepackage{comment}
\usepackage{hyperref,enumerate}
\usepackage{cleveref}
\usepackage[affil-it]{authblk}
\newtheorem{thm}{Theorem}[section]
\newtheorem{corollary}[thm]{Corollary}
\newtheorem{lem}[thm]{Lemma}

\newtheorem{obs}[thm]{Observation}

\numberwithin{subcase}{case}

\newcommand{\h}{\mathcal H}
\newcommand{\M}{\mathcal M}
\newcommand{\K}{\mathcal K}
\newcommand{\R}{\mathcal R}
\newcommand{\B}{\mathcal B}
\newcommand{\Tr}{\mathcal{T}}

\newcommand{\Cred}{\mathcal{C}_{\mathrm{red}}}
\newcommand{\Qred}{\mathcal{Q}_{\mathrm{red}}}
\newcommand{\Rred}{\mathcal{R}_{\mathrm{red}}}
\newcommand{\Hred}{\mathcal{H}_{\mathrm{red}}}
\newcommand{\Mred}{\mathcal{M}_{\mathrm{red}}}
\newcommand{\Kred}{\mathcal{K}_{\mathrm{red}}}

\newcommand{\Vred}{{V}_{\mathrm{red}}}

\newcommand{\Cblue}{\mathcal{C}_{\mathrm{blue}}}
\newcommand{\Qblue}{\mathcal{Q}_{\mathrm{blue}}}
\newcommand{\Rblue}{\mathcal{R}_{\mathrm{blue}}}
\newcommand{\Hblue}{\mathcal{H}_{\mathrm{blue}}}
\newcommand{\Mblue}{\mathcal{M}_{\mathrm{blue}}}
\newcommand{\Kblue}{\mathcal{K}_{\mathrm{blue}}}

\newcommand{\Vblue}{{V}_{\mathrm{blue}}}

\usepackage{hyphenat}
\usepackage[pagewise]{lineno}
\title{Almost partitioning 2-coloured complete 3-uniform hypergraphs into two monochromatic tight or loose cycles}
\author[1]{Sebasti\'an Bustamante\thanks{All three  authors acknowledge support by Millenium Nucleus Information and Coordination in Networks ICM/FIC RC130003.}$^{,}$\thanks{The first author was supported by CONICYT Doctoral Fellowship grant 21141116.}$^{,}$}
\affil[1]{Departmento de Ingenier\'ia Matem\'atica, Universidad de Chile, Santiago, Chile}
\author[2]{Hi\d{\^e}p H\`an$^{*,}$\thanks{The second author was supported by  the FONDECYT Iniciaci\'on grant  11150913.}$^{,}$}
\affil[2]{Departamento de Matem\'atica y Ciencia de la Computaci\'on,
Universidad de Santiago de Chile, Santiago, Chile}
\author[3]{Maya Stein$^{*,}$\thanks{The third author was supported by  FONDECYT Regular grant 1183080 and by CONICYT + PIA/Apoyo a centros cient\'ificos y tecnol\'ogicos de excelencia con financiamiento Basal, C\'odigo AFB170001.}$^{,}$}
\affil[3]{Departmento de Ingenier\'ia Matem\'atica y Centro de Modelamiento Matem\'atico (UMI 2807 CNRS), Universidad de Chile, Santiago, Chile}

\begin{document}

\maketitle
\begin{abstract}
We show that for every $\eta > 0$ there exists an integer $n_0$ such that every $2$-colouring of the $3$-uniform complete hypergraph  on $n \geq n_0$ vertices contains two disjoint monochromatic tight cycles of distinct colours that together cover
all but at most ${\eta n}$ vertices. The same result holds if tight cycles are replaced by loose cycles.
\end{abstract}

%%%%
\section{Introduction}
In recent years Ramsey type problems concerned with covering all vertices of the host graph instead of finding a small subgraph have gained popularity, both for graphs and hypergraphs. In particular, the problem of partitioning an edge-coloured complete (hyper-)graph into monochromatic cycles has received much attention. The recent surveys~\cite{FLM, Gy16}  are a good starting point for  the vast literature on the subject. 

Central to the area is an old  conjecture of Lehel  for graphs (see~\cite{Aye99}). It states that for every $n$, every  2-colouring of the edges of the complete graph~$K_n$ admits a
covering of the vertices of~$K_n$ with at most two monochromatic vertex disjoint cycles of different colours. (For technical reasons a single vertex or an edge count as a cycle.) This was confirmed for large~$n$ in~\cite{LRS98, Allen}, and for all~$n$ by Bessy and Thomass\'e~\cite{BT10}.  
There is plenty of activity in determining the number of monochromatic cycles needed if $K_n$ is coloured with more than two colours. It is known that this number is independent of $n$, but the best known lower and upper bounds, $r+1$ and $100 r \log r$, respectively, leave a considerable gap~\cite{P14,GRSS06}.

For hypergraphs much less is known. The problem transforms in the obvious way to hypergraphs, considering $r$-edge-colourings of the $k$-uniform $\K_n^{(k)}$ on $n$ vertices. However, there are several notions of hypergraph cycles and 
 referring to~\cite{FLM} for other results, we concentrate here on loose and tight cycles (see the next section for their definitions).

For loose cycles the problem was studied in~\cite{GS13,Sar14} and  the best bound, due to S\'ark\"ozy~\cite{Sar14}, shows that at most $50 rk\log(rk)$ disjoint loose monochromatic cycles are sufficient for partitioning $\K_n^{(k)}$, a number that is independent of $n$.
In general even for the case $k=3$ the problem is wide open. For example, it is unknown whether one can cover almost all vertices with two disjoint monochromatic loose cycles.

Concerning the more restrictive notion of tight cycles the situation is even worse and to our best knowledge nothing is known. Gy\'arf\'as~\cite{Gy16} conjectures that there is $c=c(r,k)$ such that every $r$-coloured $\K_n^{(k)}$ has a partition into at most $c$ monochromatic tight cycles. This, however, is open even for the ``easiest'' case of $3$-uniform hypergraphs and two colours.
Our main result establishes an approximate version for this case.

\begin{thm}\label{3unif}
	For every $\eta > 0$ there exists $n_0$ such that if $n \geq n_0$ then  every 2-coloring of the edges of the complete 3-uniform hypergraph $\K_n^{(3)}$ admits  two vertex-disjoint monochromatic tight cycles of distinct colours, 
	which cover all but at most $\eta n$ vertices.
	
	Moreover, we can choose the parity of the length of each of the cycles.
\end{thm}

We might be interested in choosing the parity of the cycles for the following reason. If $\ell$ is even, then
 any 3-uniform tight cycle on $\ell$ edges contains a loose cycle.  
Hence, we can deduce that an analogue of Theorem~\ref{3unif} holds for loose cycles.

\begin{corollary}
	For every $\eta > 0$ there exists $n_0$ such that if $n \geq n_0$ then  every 2-colouring of the edges of the complete 3-uniform hypergraph $\K_n^{(3)}$ admits  two vertex-disjoint monochromatic loose cycles, of distinct colours, 
	which cover all but at most $\eta n$ vertices.
\end{corollary}

We believe that the error term $\eta n$ in Theorem~\ref{3unif}  can be improved and that 
every 2-colouring of the edges of $\K_n^{(3)}$ admits two disjoint monochromatic tight cycles which cover all but at most a  constant number $c$ of vertices (for some $c$ independent of~$n$).

The proof of Theorem~\ref{3unif}  is inspired by the work of Haxell et~al.~\cite{HLPRRSS06,HLPRRS09} and relies on an application
of the hypergraph regularity lemma \cite{FR02}. This reduces the problem at hand to that of finding, in any 2-colouring of  the edges of an almost complete  3-uniform hypergraph, 
two disjoint \emph{monochromatic connected matchings} which cover almost all vertices. 

Here, as usual, a {\em matching}~in  a hypergraph $\h$ is a set of pairwise disjoint edges. Such a matching  $\M\subset \h$ is called {\em connected} 
if  between every pair of edges  $e,f\in\M$ there is a {\em pseudo-path} in $\h$ connecting $e$ and $f$, that is, 
a sequence $(e_1,\dots,e_p)$ of not necessarily distinct edges of $\h$ such that $e=e_1, f=e_p$ and $|e_i \cap e_{i+1}|=2$ for each $i \in [p-1]$. 
While the pseudo-paths  may use vertices outside $V(\M)$ we mean by \emph{vertices covered by~$\M$} those in~$V(\M)$ only.
Moreover, two connected matchings are \emph{disjoint} if the sets of covered vertices are.
Finally, a connected matching~$\M$ is \emph{monochromatic} if all edges in~$\M$ and all edges on the connecting paths have the same colour. 
 
The main contribution of our paper is the following result, which might be of independent interest.

\begin{thm}\label{dense_mat}
For all  $\gamma > 0$ there is a $t_0$ such that the following holds.  Given a $3$-uniform hypergraph $\h$ with $t>t_0$ vertices and at least
$(1-\gamma) \binom{t}{3}$ edges. Then any  2-colouring of the edges of  $\h$ admits  two disjoint monochromatic connected matchings covering at least $(1-290\gamma^{\frac 16})t$ vertices of~$\h$.
\end{thm}
 
We prove Theorem~\ref{dense_mat} in Section~\ref{monoxmatchings}. 
In Section~\ref{hregularity}, we introduce the regularity lemma for hypergraphs and state an embedding result from~\cite{HLPRRS09}.  The proof of
Theorem~\ref{3unif} will  be given in Section~\ref{mainproof}.

\section{Monochromatic connected matchings}\label{monoxmatchings}

Before giving the proof of Theorem~\ref{dense_mat} we introduce some notation and auxiliary results. 

Let~$\h$ denote a $k$-uniform hypergraph, that is, a pair $\h=(V,E)$ with finite vertex set $V=V(\h)$ and edge set $E=E(\h)\subset \binom Vk$, where $\binom Vk$ denotes the set
of all $k$-element sets of $V$. Often  $\h$ will be identified with its edges, that is, $\h\subset\binom Vk$ and for an edge $\{x_1,\dots,x_k\}\in \h$ we often omit brackets and commas, thus, write $x_1\dots x_k$ only.
A $k$-uniform hypergraph  $\mathcal C$ is called an \emph{$\ell$-cycle} 
if there is a cyclic ordering of the vertices of $\mathcal C$ such that every edge consists of $k$ consecutive vertices, every vertex is contained in an edge and two consecutive edges 
(where the ordering of the edges is inherited by the ordering of the vertices) intersect in exactly $\ell$ vertices. For $\ell = 1$ we call the cycle \emph{loose} whereas the cycle is called \emph{tight} if $\ell = k-1$.
% (and we do not consider other values of $\ell$).

A \emph{tight path} is a tight cycle from which one vertex and all incident edges are deleted.  The \emph{length} of a path, a pseudo-path or a cycle is the number of edges it contains.
As above, two edges in $\h$ are connected if there is a pseudo-path connecting them.
Connectedness is an equivalence relation on the edge set of $\h$ 
and the equivalence classes are called {\em connected components}.

All hypergraphs $\h$ considered from now on are $3$-uniform. 
We will need the following result concerning the existence of perfect matchings in 3-uniform hypergraphs with high minimum vertex degree.
\begin{thm}[\cite{HPS09}]\label{min_deg_mat}
	For all $\eta > 0$ there is a $n_0 = n_0(\eta)$ such that for all $n > n_0$, $n \in 3\mathbb{Z}$, the following holds. 
	Suppose $\h$ is a $3$-uniform hypergraph on $n$ vertices such that every vertex is contained in
	at least $\left( \frac{5}{9} + \eta \right ) \binom{n}{2}$ edges. Then $\h$ contains a perfect matching.
\end{thm}
Denote by $\partial \h$ the {\em shadow} of $\h$, that is, the set of all pairs $xy$ for which there exists $z$ such that $xyz \in \h$. For a vertex $x$ in a hypergraph $\h$, let $N_\h(x) = \{ y: xy \in \partial \h \}$. For two vertices $x,y$, let $N_\h(x,y) = \{ z: xyz \in \h \}$. Note that if $y \in N_\h(x)$ (equivalently, $x \in N_\h(y)$) then $N_\h(x,y) \neq \emptyset$. We call all such pairs $xy$ of vertices {\em active}.
\
\begin{lem}[\cite{HLPRRSS06}, Lemma 4.1]\label{subh_dense}
Let $\gamma > 0$ and let $\h$ be a $3$-uniform hypergraph on $t_\h$ vertices and at least $(1-\gamma)\binom{t_\h}{3}$ edges. 
Then~$\h$ contains a subhypergraph $\K$ on $t_\K\geq (1- 10 \gamma ^{1/6})t_\h$ vertices such that every vertex $x$ of $\K$ is in an active pair of $\K$ and for all active pairs $xy$ we have $|N_{\K}(x,y)| \geq (1-  10 \gamma ^{1/6})t_\K$.
\end{lem}

We are now ready to prove Theorem \ref{dense_mat}.

\begin{proof}[Proof of Theorem \ref{dense_mat}]
For given $\gamma>0$ let $\delta = 10 \gamma ^{1/6}$ and apply Theorem~\ref{min_deg_mat} with $\eta=5/36$ to obtain $n_0$.
We choose $t_0= \max\left\{\frac 2\delta, \frac{n_0}{27\delta}\right\}$.

Suppose we are given a 2-coloured $3$-uniform hypergraph $\h=\Hred\cup \Hblue$ on $t_\h>t_0$ vertices and 
$(1-\gamma) \binom{t_\h}{3}$ edges. Apply  Lemma \ref{subh_dense}~to $\h$ with parameter~$\gamma$ to obtain 
$\K$, $t:=t_\K$ with the properties stated in the lemma. Observe that at most $\delta t_\h$ vertices of $\h$ are not vertices of $\K$. We wish to find two monochromatic connected matchings covering all but at most $28\delta t\leq 28\delta t_\h$ vertices of $\K$.

Since  every vertex is in an active pair in $\K$, we have 
\begin{align}\label{eq:largeNK}|N_\K(x)|  \geq (1-\delta)t \quad\text{for all }x \in V(\K).\end{align} Let $\K = \Kred \cup \Kblue$ be the colouring of $\K$ inherited from $\h$.
Then a monochromatic component $\mathcal C$ of $\K$ is a connected component of $\Kred$ or~$\Kblue$.

\begin{obs}[\cite{HLPRRS09}, Observation 8.2]\label{obs_neig_dense}
	For every vertex $x \in V(\K)$  there exists a monochromatic component $\mathcal C_x$ such that $ |N_{\mathcal C_x}(x)| \geq (1- \delta)t.$
\end{obs}

For each $x \in V(\K)$ choose  one component $\mathcal C_x$ as in Observation~\ref{obs_neig_dense}. Let $R= \{ x \in V(\K): \mathcal C_x \text{ is red} \}$ and $B = \{ x \in V(\K): \mathcal C_x \text{ is blue} \}$, and note that these two sets partition $V(\K)$. 

\begin{obs}[\cite{HLPRRS09}, Observation 8.4]\label{obs_neig_dense2}
	If $|R| \geq 6\delta t$ (or $|B| \geq 6\delta t$, respectively), then there is a red component $\R$ (a blue component $\B$) such that $\mathcal C_x = \R$ ($\mathcal C_x = \B$) for all but at most $2\delta t$ vertices $x \in R$ ($x \in B$).
\end{obs}

Set $\Vred:= \{ x \in R: C_x = \R \}$ if $|R| \geq 6\delta t$, and set $\Vblue := \{ x \in B : C_x = \B \}$ if $|B| \geq 6\delta t$. Otherwise, define $\Vred$ (or $\Vblue$, respectively) as the empty set. Our aim is to find two differently coloured disjoint connected matchings in~$\K$ that together cover all but $16\delta t\leq 28\delta t-12\delta t$ vertices of $\Vred \cup \Vblue$. 

We start by choosing a connected matching of maximal size in $\R\cup\B$. This matching decomposes into two disjoint  monochromatic connected matchings, $\Mred\subset \R$ and $\Mblue\subset \B$, which together cover as many vertices as possible. Let $\Vred' = \Vred\setminus V(\Mred\cup \Mblue)$ and $\Vblue' = \Vblue\setminus V(\Mred\cup \Mblue)$. We may assume that $\Vred'$ or $\Vblue'$ has at least $12\delta t$ vertices, as otherwise we are done. By symmetry we may assume that 
\begin{equation}\label{V_R'}
|\Vred'|\geq 8\delta t.
\end{equation}

Observe that there is no pair $xy$ with $x\in \Vred'$ and $y\in \Vblue'$ such that $xy \in \partial \R \cap \partial \B$. Indeed, any such edge $xy$ constitutes an active pair (by Lemma \ref{subh_dense}) and as $|\Vred'| > \delta t+2$, there must be a vertex $z\in \Vred'$ such that $xyz$ is an edge of $\K$. This contradicts the maximality of the matching $\Mred\cup \Mblue$.

Next, we claim that 
\begin{equation}\label{V_B'}
|\Vblue'|\leq 2\delta t.
\end{equation}

Assume otherwise. Then, Observation~\ref{obs_neig_dense} and the choice of the set $\Vred$ implies that the number of edges between $\Vred'$ and $\Vblue'$  that belong to $\partial \R$ is at least 
\[|\Vred'|\cdot (|\Vblue'| -\delta t)>\frac 12 |\Vred'|\cdot|\Vblue'|.\] Similarly,  there are at least $|\Vblue'|\cdot (|\Vred'| -\delta t)>\frac 12 |\Vred'|\cdot|\Vblue'|$  edges between $\Vred'$ and $\Vblue'$ that belong to $\partial \B$. As there is no edge $xy$ with $x\in \Vred'$ and $y\in \Vblue'$ such that $xy \in \partial \R \cap \partial \B$, we have more than $|\Vred'|\cdot|\Vblue'|$ edges from  $\Vred'$ to $\Vblue'$. This yields a contradiction and~\eqref{V_B'} follows.

Because of the maximality of $\Mred\cup \Mblue$, each edge having all its vertices in $\Vred'$ is blue. Fix one such edge $xyz$, which exists because of~\eqref{V_R'}. Obtain $\Vred''$ from $\Vred'$ by deleting the at most $\delta t$ vertices $w$ with $wx\notin \partial \R$. Consider any edge $x'y'z'$ with $x',y',z'\in\Vred''$, which also exists because of~\eqref{V_R'}. As the pairs $xy$, $xx'$, $x'y'$ are all active, and $|\Vred''|>3\delta t$, there is a vertex $v\in \Vred''$ that forms an edge with each of the three pairs, thus giving a pseudo-path in $\K[\Vred'']$ from $xyz$ to $x'y'z'$. Denote by~$\B''$ the blue component of $\K[\Vred'']$ that contains~$xyz$, and let $\B'$ be obtained from $\B''$ by deleting at most two vertices and all incident edges,  so that $|V[\B']|$ is a multiple of three. Then, by~\eqref{V_R'}, we have
\begin{equation}\label{V_calB'}
|V[\B']|\geq |\Vred'|-\delta t-2\geq 6\delta t.
\end{equation}

Let $x \in V[\B']$ be given. At least $|V[\B']|-\delta t$ vertices $y\in V[\B']$ are such that $xy \in \partial \R$, and, for each such $y$ there are at least $|V[\B']|-\delta t$ vertices $z \in V[\B']$ such that $xyz \in \B'$. So, the total number of hyperedges of $\B'$ that contain~$x$ is at least \[\frac 12 (|V[\B']|-\delta t)^2 \ \geq  \ \frac{25}{36} {|V[\B']|\choose 2}.\] Thus, Theorem \ref{min_deg_mat} with $\eta=\frac{5}{36} $ yields a perfect matching $\Mblue'$ of $\B'$. 

At this point, we have three disjoint monochromatic connected matchings, one in red ($\Mred\subset \R$) and two in blue ($\Mblue\subset \B$ and $\Mblue' \subset \B'$). Together, these matchings cover all but at most $3\delta t +2$ vertices of $\Vred\cup \Vblue$ (by~\eqref{V_B'} and by~\eqref{V_calB'}). 

Our aim is now to dissolve the blue matching $\Mblue$, and cover its vertices by new red edges,  leaving at most $6\delta t$ vertices uncovered. In order to do so, let us first understand where the edges of $\Mblue$ lie.

An edge in $\K$ is called {\em good} if there is one pair of its vertices contained in $\partial \R$ and another contained in $\partial \B$.
 Notice that every good red edge is contained in $\R$ and  every good blue edge is contained in $\B$. 

First, we claim  that for every edge $uvw\in \Mblue$, 
\begin{equation}\label{MBcapVB}
|\{u,v,w\} \cap \Vblue| \leq 1.
\end{equation}
Indeed, otherwise there is an edge $uvw\in \Mblue$ with $u,v \in \Vblue$. By the definition of $\B$, and by~\eqref{V_R'}, there is an active edge $ua \in \partial \B$ with $a \in \Vred'$. As $ua$ is an active pair, as $a$ has very large degree in $\partial \R$, and by~\eqref{V_R'},  there is an edge $uab \in \K$ with $b \in \Vred'$ such that $ab \in \partial \R$. Hence $uab$ is a good edge. Similarly, there 
is a good edge $vcd$, with $c,d\in \Vred'\setminus\{a,b\}$. Remove the edge $uvw$ from $\Mblue$ and add edges $uab$ and $vcd$ to either $\Mred$ or $\Mblue$, according to their colour. The resulting matching covers more vertices than the matching  $M_A\cup \Mblue$, a contradiction. This proves~\eqref{MBcapVB}.

Next, we claim  that there is no edge $uvw\in \Mblue$ with
\begin{equation}\label{MBcapVB=1}
|\{u,v,w\} \cap \Vblue| = 1.
\end{equation}
Assume otherwise. Then there is an edge $uvw\in \Mblue$ with $u \in \Vblue$ and $v,w\in \Vred$. As in the proof of~\eqref{MBcapVB}, we can cover $u$ with a good edge $uab$ such that $a,b \in \Vred'$. Moreover, since $vw$ is an active pair, and $v$ has  very large degree in $\partial \R$, there is an edge $vwc$ with $c \in \Vred' \setminus \{a,b\}$ and $cv \in \partial R$. 
Since $vw \in \partial \B$, the edge $vwc$ is good. So we can remove $uvw$ from $\Mblue$ and add edges $uab$ and $vwc$ to $\Mred \cup \Mblue$, thus covering three additional vertices. This gives the desired contradiction to the choice of $\Mred \cup \Mblue$, and proves~\eqref{MBcapVB=1}.

Putting~\eqref{MBcapVB} and~\eqref{MBcapVB=1} together, we know that 
for every edge $uvw \in \Mblue$ we have $u,v,w \in \Vred$.  We can assume that
$\Mblue$ contains at least two hyperedges, as otherwise we can just forget about $\Mblue$ and we are done. Consider any two edges $u_1v_1w_1, u_2v_2w_2 \in \Mblue$. As before,  there are vertices $a,b \in \Vred'$ such that edges $v_1w_1a, v_2w_2b$ are good. Now, if there is a red edge $u_1u_2c$ with $c \in \Vred'$ and $u_1c\in\partial \R$ then we can remove edges $u_1v_1w_1, u_2v_2w_2$ and add the red edge $u_1u_2c$ to $\Mred$ and edges $v_1w_1a,v_2w_2b$ to $\Mred \cup \Mblue$, according to their colour, contradicting the choice of $\Mred \cup \Mblue$. Therefore, for any choice of $u_1v_1w_1, u_2v_2w_2 \in \Mblue$, we have that
\begin{equation}\label{allblue}
\text{all edges $u_1u_2c$ with $c \in \Vred'$ and $u_1c\in\partial \R$ are blue.}
\end{equation}

Moreover, if there is a blue edge $u_1u_2x$ with $x \in \{v_1,w_1,v_2,w_2\}$ then $u_1u_2$ is an active pair. In that case, we can argue as before that an edge $u_1u_2c$  with $c \in \Vred'$ and $u_1c\in\partial \R$ exists, and by~\eqref{allblue}, this edge is blue. The existence of the blue edge $u_1u_2x$ implies that we can link $u_1u_2c$ to $\Mblue$ with a blue tight path. Thus, removing $u_1v_1w_1$ and $u_2v_2w_2$ from $\Mblue$ and adding  $v_1w_1a,v_2w_2b,u_1u_2c$ to $\Mred \cup \Mblue$ (where $a,b$ are as above), we obtain a  contradiction to the choice of $\Mred \cup \Mblue$. So, for any choice of $u_1v_1w_1, u_2v_2w_2 \in \Mblue$, we have that
\begin{equation}\label{allred}
\text{all edges $u_1u_2x$ with $x \in \{v_1,w_1,v_2,w_2\}$ are red.}
\end{equation}

We can now dissolve the edges of $\Mblue$. For this, separate each hyperedge $uvw$ in $\Mblue$ into an edge $uv$ and a single vertex $w$. Let $X$ be the set of all edges $uv$, and let $Y$ be the set of all vertices $w$ obtained in this way. Note that every $uv\in X$ is an active pair in $\K$, and therefore forms a hyperedge $uvw'$ with all but at most $\delta t$ of the vertices $w'\in Y$. Moreover, all but at most $\delta t$ of these hyperedges $uvw'$ are such that $uw'\in\partial \R$, because of the large degree $u$ has in $\partial\R$. Now we can greedily match all but at most $2\delta t$ edges $uv$ in $X$ with vertices $w'$ in $Y$ such that for every match $uvw'$ we have that $uvw'$ is an edge of $\K$ and $uw \in \partial{\R}$.

%Consider the bipartite graph on $X\cup Y$ which has an edge between $uv\in X$ and $w'$ whenever the hyperedge $uvw'$ exists in $\K$ and $uw'\in\partial \R$. Then for each $X' \subseteq X$  we have that $|N(X')| \geq  |Y| - 2\delta t \geq |X'| - 2\delta t$. So by the defect form of Hall's Theorem, there is a matching covering all but at most $2\delta t$ vertices of $|X|$. 
 
In $\K$, this corresponds to a matching $\Mred'$ covering all but at most $6\delta t$ vertices of $V(\Mblue)$. By~\eqref{allred}, all hyperedges of $\Mred'$ are red. Furthermore, since we ensured that every hyperedge in $\Mred'$ contains a pair $uw'$ that forms an edge of $\partial\R$, we know that $\Mred$ and $\Mred'$ belong to the same red component of $\K$. In other words, $\Mred\cup \Mred'$ and $\Mblue'$ are the two monochromatic connected matchings we wished to find.
\end{proof}

%%%%
\section{Hypergraph regularity}\label{hregularity}
In this section we introduce the regularity lemma for 3-uniform hypergraphs and state an embedding result from~\cite{HLPRRS09}.

%We start with graph regularity.
\paragraph{Graph regularity.} Let $G$ be a graph and let  $X,Y \subseteq V(G)$ be disjoint. 
The \emph{density} of $(X,Y)$ is   $d_G(X,Y) = \frac{e_G(X,Y)}{|X||Y|}$ where
$e_G(X,Y)$ denotes the number of edges of $G$ between $X$ and $Y$. 

The bipartite graph $G$ on the partition classes $X$ and $Y$ is called {\em $(d,\varepsilon)$-regular}, if $ |d_G(X',Y') - d| < \varepsilon$ holds for all 
$X' \subseteq X$ and $Y' \subseteq Y$ of size $|X'| > \varepsilon |X|$ and $|Y'| > \varepsilon |Y|$.
If $d = d_G(X,Y)$ we say that $G$ is {\em $\varepsilon$-regular}.

\paragraph{Hypergraph regularity.} 
Let $\h$ be a $3$-uniform hypergraph. Let $P=P^{12}\cup P^{13}\cup P^{23}$ with $V(P)\subset V(\h)$ be a tripartite graph which we also refer to as \emph{triad}. 
By $\Tr(P)$ denote the 
$3$-uniform hypergraph on $V(P)$ whose edges are the triangles of $P$.
 The \emph{density of $\h$ with respect to $P$} is 
\[d_\h(P) = \frac{|\h \cap \Tr(P)|}{|\Tr(P)|}. \]
Similarly, for a tuple $\vec{\mathcal{Q}} = (Q_1,\dots,Q_r)$ of subgraphs of $P$, we define the \emph{density of $\h$ with respect to $\vec{\mathcal{Q}}$} as $$ d_\h(\vec{\mathcal{Q}}) = \frac{|\h \cap \bigcup_{i \in [r]} \Tr(Q_i)|}{|\bigcup_{i \in [r]} \Tr(Q_i)|}. $$

Let $\alpha, \delta > 0$ and let $r>0$ be an integer. We say that $\h$ is {\em $(\alpha,\delta,r)$-regular} with respect to $P$ if, for every $r$-tuple $\vec{\mathcal{Q}} = (Q_1,\dots,Q_r)$ of subgraphs of $P$ satisfying $|\bigcup_{i \in [r]} \Tr(Q_i)| > \delta |\Tr(P)|$, we have $|d_\h(\vec{\mathcal{Q}}) - \alpha| < \delta$. If $\alpha = d_\h(P)$ we say that $\h$ is {\em $(\delta,r)$-regular} with respect to $P$, and in the same situation, we say  $P$ is {\em $(\delta,r)$-regular} (with respect to $\h$). 

If in addition the bipartite graphs $P^{12},P^{13}, P^{23}$ of an $(\alpha,\delta,r)$-regular $P=P^{12}\cup P^{13}\cup P^{23}$ are
$(1/\ell, \varepsilon)$-regular then 
we say that the pair $(\h,P)$ is an {\em $(\alpha, \delta, \ell, r, \varepsilon)$-regular complex}.

Finally, a partition 
of $V$ into $V_0\cup V_1\cup\dots\cup V_t$ is called an \emph{equipartition} if $|V_0|<t$ and $|V_1|=V_2|=\dots=|V_t|$.

We state the regularity lemma for $3$-uniform hypergraphs \cite{FR02} as  presented in \cite{RRS06}.

\begin{thm}[Regularity Lemma for $3$-uniform Hypergraphs]\label{hreglemma}
	For all $\delta, t_0 > 0$, all integer-valued functions $r=r(t,\ell)$, and all  decreasing sequences $\varepsilon(\ell) > 0$ there exist constants $T_0,L_0$ and $N_0$ 
	such that every $3$-uniform hypergraph $\h$ with at least $N_0$ vertices admits a vertex equipartition \[V(\h) = V_0 \cup V_1 \cup \dots \cup V_t  \qquad\text{ with } t_0 \leq t < T_0,\]
	  and, for each pair $i,j,$ $1 \leq i < j \leq t,$ an edge partition of the complete bipartite graph 
	\[K(V_i,V_j) = \bigcup_{k \in [\ell]} P^{ij}_{k} \qquad \text{ with } 1 \leq \ell < L_0\]
such that 
	\begin{enumerate}
		\item all graphs $P_k^{ij}$ are $(1/\ell, \varepsilon(\ell))$-regular.
		\item $\h$ is $(\delta,r)$-regular with respect to all but at most $\delta \ell ^3 t^3$ tripartite graphs $P_a^{hi} \cup P_b^{hj} \cup P_c^{ij}$.
	\end{enumerate}
	
\end{thm}

Note that the same partitions satisfy the conclusions of Theorem~\ref{hreglemma}  for the complement of $\h$ as well.  
Further, as noted in \cite{HLPRRS09} by choosing a random index $k_{ij}\in[\ell]$ for each pair $(V_i,V_j)$ Markov's inequality yields that 
with positive probability  there are less than $2\delta t^3$ chosen triads 
which fail to be $(\delta,r)$-regular. Hence one obtains the following.
\begin{obs}\label{obs:goodgraphs}
In the partition produced by Theorem \ref{hreglemma} there is a family $\mathcal{P}$ of bipartite graphs $P^{ij} = P^{ij}_{k_{ij}}$ with vertex classes $V_i,V_j$, where $1 \leq i < j \leq t$, such that $\h$ is $(\delta,r)$-regular with respect to all but at most $2\delta t^3$ tripartite graphs $P^{hi} \cup P^{hj} \cup P^{ij}$.
\end{obs}

We end this section with a result from~\cite{PRRS06} and~\cite{HLPRRS09} which allows embedding tight paths in regular complexes. 
In the following, an $S$-avoiding tight path is one which does not contain any vertex from $S$.
(Note that although Lemma 4.6 from~\cite{HLPRRS09} is stated slightly differently, its proof actually yields the version below.)

\begin{lem}[\cite{HLPRRS09}, Lemma 4.6]\label{pathcomplex}
	For each $\alpha \in (0,1)$ there exist  $\delta_1  > 0$ and sequences $r(\ell)$, $\varepsilon(\ell)$, 
	and $n_1(\ell)$, for $\ell\in\mathbb N$, with the following property. \\ For each $\ell \in \mathbb N$, and each $\delta\leq \delta_1$,
	if $(\h,P)$ is a $(d_\h(P),\delta, \ell, r(\ell), \varepsilon(\ell))$-complex with $d_\h(P) \geq \alpha$ and 
	all of the three  vertex classes of $P$ have the same size $n > n_1(\ell)$, then there is a subgraph 
	$P_0$ on at most $27 \sqrt \delta n^2 / \ell$ edges of $P$ such that, for all ordered pairs of disjoint edges 
	$(e,f) \in (P \setminus P_0) \times (P \setminus P_0)$ there is $m = m(e,f) \in [3]$ such that the following holds. For every $S \subseteq V(\h) \setminus (e \cup f)$ with 
	$|S| < n/(\log n)^2$,  and for each $\ell$ with $3 \leq \ell \leq (1-2\delta ^{\frac{1}{4}})n$, there is an $S$-avoiding tight path from $e$ to $f$ of length $3 \ell + m$ in $\h$.
\end{lem}

\section{Proof of Theorem \ref{3unif}}\label{mainproof}

We follow a procedure suggested by \L uczak in \cite{Luc99} for graphs and used for tight cycles in $3$-uniform hypergraphs in \cite{HLPRRS09}.
\begin{proof}[Proof of Theorem~\ref{3unif}]
For given    $\eta>0$ we apply Lemma~\ref{dense_mat} with $\gamma=(\eta/580)^6$ to obtain $t_0$.
With foresight apply Lemma~\ref{pathcomplex} with $\alpha=1/2$ to obtain $\delta_1$, and sequences $r(\ell)$, $\varepsilon(\ell)$, 
and $n_1(\ell)$. Finally, apply Theorem \ref{hreglemma} with $t_0$, $r(t,\ell)=r(\ell)$, $\varepsilon(\ell)$, $n_1(\ell)$ and $\delta=\min\{\delta_1/2,\gamma/58, (\eta/16)^4\}$
to obtain constants $T_0$, $L_0$ and $N_0$. 

Given a 2-colouring $\K_n=\h_{\mathrm{red}}\cup \h_{\mathrm{blue}}$ of the 3-uniform complete hypergraph $\K_n$ on  $n>N_0$ vertices.
Apply   Theorem \ref{hreglemma} with the chosen constants to $\h_{\mathrm{red}}$ to  
obtain partitions 
\[V(\K_n)=V_0\cup V_1\cup\dots\cup V_t \quad\text{and}\quad K(V_i,V_j)=\bigcup_{k \in [\ell]} P^{ij}_{k}, 1\leq i<j\leq t\]
with $t_0\leq t< T_0$, and $\ell<L_0$ which satisfy the properties detailed in Theorem~\ref{hreglemma}.
The partitions satisfy the same properties for  $\h_{\mathrm{blue}}$ as it  is the complement hypergraph of $\h_{\mathrm{red}}$. 

 Observation~\ref{obs:goodgraphs} then yields  a family of $(1/\ell, \varepsilon)$-regular bipartite graphs $P^{ij} = P^{ij}_{k_{ij}}$, one for  each pair 
$(V_i,V_j)$, $1 \leq i < j \leq t$, such that $\h_{\mathrm{red}}$ (and thus also $\h_{\mathrm{blue}}$) is $(\delta, r)$-regular with respect to all but at most $2\delta t^3$ tripartite graphs $P^{ijk} = P^{ij} \cup P^{ik} \cup P^{jk}$. 
We use this family to construct the reduced hypergraph $\R$ which has the vertex set $[t]$ and the edge set consisting of all triples $ijk$ such that $P^{ijk}$ is $(\delta,r)$-regular.
Further, colour the edge $ijk$ red if $d_{\h_{\mathrm{red}}}(P^{ijk}) \geq 1/2$ and blue otherwise.  Then we have a 2-colouring of $\R = \Rred \cup \Rblue$, where $\R$ has at least $\binom{t}{3} - 2\delta t^3 > (1 - \gamma)\binom{t}{3}$ edges.

Since $t\geq t_0$  Lemma \ref{dense_mat} yields two disjoint monochromatic  connected matchings 
$\M_{\mathrm{red}}$ and $\M_{\mathrm{blue}}$ 
which cover all but at most $290\gamma^{\frac 16} t\leq \eta t/2$ vertices of $\R$ and in what  follows we will 
turn these  connected matchings into disjoint monochromatic tight cycles in $\K_n$.

\medskip
We start by choosing a red pseudo-path $\Qred=(e_1,\dots,e_p)\subset\Rred$ which contains the matching  $\Mred$. This is possible since $\Mred$ is a connected matching, and so, consecutive matching edges $g_s, g_{s+1}\in \Mred$ are connected by a red pseudo-path of length at most $\binom t3$.
 The concatenation of these paths then forms a $\Qred$ as desired. In the same manner, choose a 
 blue pseudo-path $\mathcal{Q}_{\mathrm{blue}}=(e'_1,\dots,e'_q)\subset\Rblue$ containing the matching~$\Mblue$. Note that although $\Mred$ and $\Mblue$ are disjoint, the two paths $\Qred$ and $\Qblue$ 
 may have vertices in common.

 The general idea is to find long tight cycles in different colours is as follows. For each edge $\{i,j,k\}=e_s\in \Qred$ let $P^s$ denote the triad $P^{ij}\cup P^{ik}\cup P^{jk}$ on the partition classes $V_i\cup V_j\cup V_k$ and recall that $P^s$ is $(\delta,r)$-regular (with respect to $\Hred$).
 Lemma~\ref{pathcomplex} guarantees that one can find long tight paths in the complex $(\Hred,P^s)$  for each matching edge $e_s\in\Mred\subset \Qred$ 
which exhaust almost all vertices of $V_i\cup V_j\cup V_k$. 
 Using the connectedness of $\Qred$  and Lemma~\ref{pathcomplex}, we then want to
connect these long tight paths by short tight paths, hence obtain a tight cycle $\Cred$ which covers almost all vertices of $\K_n$ spanned by $\Mred$. We wish to do the same with $\Qblue$ to obtain
a tight cycle $\Cblue$ which covers almost all vertices of $\K_n$ spanned by $\Mblue$. The two cycles $\Cred$ and $\Cblue$ then exhaust most of the vertices of $\K_n$.

To keep the two cycles disjoint, however, the strategy will be slightly less straightforward. First, we will find two disjoint short tight cycles 
$\Cred'$ and $\Cblue'$ in $\K_n$ visiting all triads $P^s$ corresponding to  edges  $e_s\in\Mred$ and $P'^s$ corresponding to $e'_s\in\Mblue$, respectively. 
Then, for each edge  $e_s\in\Mred$, and each edge $e'_s\in\Mblue$, we will replace parts of the cycles $\Cred'$, and of $\Cblue'$, i.e., paths corresponding to $e_s$ and $e'_s$ by  long tight paths as mentioned above. 
We now give the details of this idea.

For each $s=1,\dots, p$, apply Lemma~\ref{pathcomplex} to the complex $(\Hred,P^s)$  to obtain  the subgraph $P^s_0\subset P^s$  of ``prohibited'' edges
and let  \[B_s= (P^s\setminus P^s_0)\cap (P^{s+1}\setminus P^{s+1}_0)\]
which is a   bipartite graph on the partition classes $V_i\cup V_j$ where $\{i,j\}=e_s\cap e_{s+1}$. We choose mutually distinct edges $f_s, g_s\in B_s$, $s\in[p-1]$ which  is possible due to the restriction on 
$|P_0^s|$ provided $n$ is sufficiently large. Using  Lemma~\ref{pathcomplex} we then find a short tight cycle $\Cred'$ by
concatenating disjoint paths each of length at most 12 between $f_1$ and $f_2$, $f_2$ and $f_3$  $\dots$ between $f_{p-1}$ and $g_{p-1}$ and backwards between $g_{p-1}$ and $g_{p-2}$  $\dots$ and
finally between $g_1$ and  $f_1$. 
Note that the lemma allows the paths to be $S$-avoiding for any vertex set $S$ of size  $|S|<n'/(\log n')^2$ where $n'$ is the size of the partition classes. 
Therefore, to guarantee the disjointness of the paths, we simply choose $S$ to be the vertices of the paths constructed so far which 
has size at most $24p$, i.e., independent of $n'>n/2t$. In the same way choose a short tight cycle $\Cblue'$ disjoint
from $\Cred'$ by including $V(\Cred')$ to  $S$ in the applications of Lemma~\ref{pathcomplex}.

Let $S'=V(\Cred')\cup V(\Cblue')$ which satisfies $|S'|<n'/(\log n')^2$. It is easy to see that for each $e_s\in \Mred$ there are two non-prohibited edges in $P^s$,  connected by a subpath of $\Cred'$ 
which is entirely contained in $(\Hred,P^s)$. Hence, by Lemma~\ref{pathcomplex} we can replace this short path by an $S$-avoiding path in $(\Hred,P^s)$ which covers all but at most $4\delta^{1/4}n'$
vertices and having any desired parity. Doing this for all $e_s\in \Mred$ and all $e'_s\in \Mblue$ and noting that $n'\leq n/t$ we obtain two monochromatic disjoint tight cycles  which cover all but at most 
\begin{align*}\big(|\Mred|+|\Mblue|\big)4\delta^{1/4}n'+& \big(|V(\R)|-|\Mred|+|\Mblue|\big)n'+ |V_0| \\
&\leq \frac 14\eta n+\frac 12\eta n+t \leq \eta n \end{align*}
vertices of $\K_n$, and have any desired parity. This finishes the proof of the theorem.
\end{proof}

\end{document}